\documentclass{amsart}
\usepackage{amsfonts}

\newtheorem{theorem}{Theorem}
\theoremstyle{plain}

\newtheorem{corollary}{Corollary}

\newtheorem{lemma}{Lemma}

\newtheorem{proposition}{Proposition}

\numberwithin{equation}{section}
\input{tcilatex}

\begin{document}
\title[Hermite-Hadamard Type Inequalities]{Hermite-Hadamard type
inequalities for functions whose derivatives are $(\alpha ,m)$-convex}
\author{\.{I}mdat \.{I}\c{s}can}
\address{Department of Mathematics, Faculty of Science and Arts,\\
Giresun University, 28100, Giresun, Turkey.}
\email{imdat.iscan@giresun.edu.tr}
\date{February 26, 2012}
\subjclass[2000]{26A51, 26D15}
\keywords{$(\alpha ,m)$-convex functions,Hermite-Hadamard's inequality, H%
\"{o}lder's integral inequality}

\begin{abstract}
In this paper several inequalities of the right-hand side of
Hermite-Hadamard inequality are obtained for the class of functions whose
derivatives in absolutely value at certain powers are $(\alpha ,m)$%
-convex.Some applications to special means of positive real numbers are also
given.
\end{abstract}

\maketitle

\section{Introduction}

\bigskip

Let $f:I\subset \mathbb{R\rightarrow R}$ be a convex function defined on the
interval $I$ of real numbers and $a.b\in I$ with $a<b$, then%
\begin{equation}
f\left( \frac{a+b}{2}\right) \leq \frac{1}{b-a}\dint\limits_{a}^{b}f(x)dx%
\leq \frac{f(a)+f(b)}{2}\text{.}  \label{1-1}
\end{equation}

This doubly inequality is known in the literature as Hermite-Hadamard
integral inequality for convex functions.

In \cite{M93} Mihe\c{s}an introduced the class of $(\alpha ,m)$-convex
functions as the following:

The function $f:\left[ 0,b\right] \mathbb{\rightarrow R}$, $b>0$, is said to
be $(\alpha ,m)$-convex where $(\alpha ,m)\in \left[ 0,1\right] ^{2}$, if we
have%
\begin{equation*}
f\left( tx+m(1-t)y\right) \leq t^{\alpha }f(x)+m(1-t^{\alpha })f(y)
\end{equation*}%
for all $x,y\in \left[ 0,b\right] $ and $t\in \left[ 0,1\right] $.

It can be easily that for $(\alpha ,m)\in \left\{ (0,0),(\alpha
,0),(1,0),(1,m),(1,1),(\alpha ,1)\right\} $one obtains the following classes
of functions: increasing, $\alpha $-starshaped, starshaped, $m$-convex,
convex, $\alpha $-convex.

Denote by $K_{m}^{\alpha }(b)$ the set of all $(\alpha ,m)$-convex functions
on $\left[ 0,b\right] $ for which $f(0)\leq 0$. For recent results and
generalizations concerning $m$-convex and $(\alpha ,m)$-convex functions
(see \cite{BOP08,BPR06,OAK11,OAS10,OKS10,OSS11,SOS10,SSOR09}).

In \cite{DA98} Dragomir and Agarwal established the following result
connected with the right-hand side of (\ref{1-1}).

\begin{theorem}
Let $f:I\subset \mathbb{R\rightarrow R}$ be a differentiable mapping on $%
I^{\circ }$, where $a,b\in I$ with $a<b.$If $\left\vert f^{\prime
}\right\vert $ is convex on $[a,b]$, then the following inequality holds:%
\begin{equation}
\left\vert \frac{f(a)+f(b)}{2}-\frac{1}{b-a}\dint\limits_{a}^{b}f(x)dx\right%
\vert \leq \frac{b-a}{8}\left[ \left\vert f^{\prime }(a)\right\vert
+\left\vert f^{\prime }(b)\right\vert \right] .  \label{1-2}
\end{equation}
\end{theorem}

In \cite{SSOR09}, the following inequality of Hermite-Hadamard type for $%
(\alpha ,m)$-convex functions holds:

\begin{theorem}
\label{A.1}Let $f:\left[ 0,\infty \right) \rightarrow \mathbb{R}$ be an $%
(\alpha ,m)$-convex function with $(\alpha ,m)\in \left( 0,1\right] ^{2}.$
If $0\leq a<b<\infty $ and $f\in L[a,b],$ then one has the inequality:%
\begin{equation}
\frac{1}{b-a}\dint\limits_{a}^{b}f(x)dx\leq \min \left\{ \frac{f(a)+\alpha
mf\left( \frac{b}{m}\right) }{\alpha +1},\frac{f(b)+\alpha mf\left( \frac{a}{%
m}\right) }{\alpha +1}\right\} .  \label{1-3}
\end{equation}
\end{theorem}

In \cite{BOP08} the following Hermite-Hadamard type inequalities for $m-$
and $(\alpha ,m)$-convex functions were obtained.

\begin{theorem}
Let $I$ be an open real interval such that $\left[ 0,\infty \right) \subset
I $. Let $f:I\mathbb{\rightarrow R}$ be a differentiable mapping on $I$ such
that $f^{\prime }\in L[a,b],$ where $0\leq a<b<\infty .$ If $\left\vert
f^{\prime }\right\vert ^{q}$ is $m$-convex on $[a,b]$, for some fixed $m\in
\left( 0,1\right] $ and $q\in \left( 1,\infty \right) ,$ then%
\begin{eqnarray}
\left\vert \frac{f(a)+f(b)}{2}-\frac{1}{b-a}\dint\limits_{a}^{b}f(x)dx\right%
\vert &\leq &\frac{b-a}{4}\left( \frac{q-1}{2q-1}\right) ^{\frac{q-1}{q}%
}\left( \mu _{1}^{\frac{1}{q}}+\mu _{2}^{\frac{1}{q}}\right)  \label{B1} \\
&\leq &\frac{b-a}{4}\left( \mu _{1}^{\frac{1}{q}}+\mu _{2}^{\frac{1}{q}%
}\right) \   \notag
\end{eqnarray}%
where%
\begin{eqnarray*}
\mu _{1} &=&\min \left\{ \frac{\left\vert f^{\prime }(a)\right\vert
^{q}+m\left\vert f^{\prime }\left( \frac{a+b}{2m}\right) \right\vert ^{q}}{2}%
,~\frac{\left\vert f^{\prime }(\frac{a+b}{2})\right\vert ^{q}+m\left\vert
f^{\prime }\left( \frac{a}{m}\right) \right\vert ^{q}}{2}\right\} , \\
\mu _{2} &=&\min \left\{ \frac{\left\vert f^{\prime }(b)\right\vert
^{q}+m\left\vert f^{\prime }\left( \frac{a+b}{2m}\right) \right\vert ^{q}}{2}%
,~\frac{\left\vert f^{\prime }(\frac{a+b}{2})\right\vert ^{q}+m\left\vert
f^{\prime }\left( \frac{b}{m}\right) \right\vert ^{q}}{2}\right\} .
\end{eqnarray*}
\end{theorem}

\begin{theorem}
Let $I$ be an open real interval such that $\left[ 0,\infty \right) \subset
I $. Let $f:I\mathbb{\rightarrow R}$ be a differentiable mapping on $I$ such
that $f^{\prime }\in L[a,b],$ where $0\leq a<b<\infty .$ If $\left\vert
f^{\prime }\right\vert ^{q}$ is $(\alpha ,m)$-convex on $[a,b]$, for some
fixed $\alpha ,m\in \left( 0,1\right] $ and $q\in \left[ 1,\infty \right) ,$
then%
\begin{eqnarray}
&&\left\vert \frac{f(a)+f(b)}{2}-\frac{1}{b-a}\dint\limits_{a}^{b}f(x)dx%
\right\vert \ \ \leq \frac{b-a}{2}\left( \frac{1}{2}\right) ^{1-\frac{1}{q}}
\label{1-4} \\
&&\ \times \min \left\{ \left( \nu _{1}\left\vert f^{\prime }(a)\right\vert
^{q}+m\nu _{2}\left\vert f^{\prime }(\frac{b}{m})\right\vert ^{q}\right) ^{%
\frac{1}{q}},~\left( \nu _{1}\left\vert f^{\prime }(b)\right\vert ^{q}+m\nu
_{2}\left\vert f^{\prime }(\frac{a}{m})\right\vert ^{q}\right) ^{\frac{1}{q}%
}\right\}  \notag
\end{eqnarray}

where%
\begin{equation*}
\nu _{1}=\frac{1}{\left( \alpha +1\right) \left( \alpha +2\right) }\left[
\alpha +\left( \frac{1}{2}\right) ^{\alpha }\right]
\end{equation*}%
and%
\begin{equation*}
\nu _{2}=\frac{1}{\left( \alpha +1\right) \left( \alpha +2\right) }\left[ 
\frac{\alpha ^{2}+\alpha +2}{2}-\left( \frac{1}{2}\right) ^{\alpha }\right] .
\end{equation*}
\end{theorem}

The main aim of this paper is to establish new inequalities of
Hermite-Hadamard type for the class of functions whose derivatives in
absolutely value at certain powers are $(\alpha ,m)$-convex.

\section{Inequalities for functions whose derivatives are $(\protect\alpha %
,m)$-convex}

In order to prove our main resuls we need the following lemma:

\begin{lemma}
\label{2.1}Let $f:I\subset \mathbb{R\rightarrow R}$ be a differentiable
mapping on $I^{\circ }$, $a,b\in I$ with $a<b$. If $f^{\prime }\in L[a,b]$
and $\lambda ,\mu \in \left[ 0,\infty \right) ,~\lambda +\mu >0,$ then the
following equality holds:%
\begin{equation*}
\frac{\lambda f(a)+\mu f(b)}{\lambda +\mu }-\frac{1}{b-a}\dint%
\limits_{a}^{b}f(x)dx=\frac{b-a}{\lambda +\mu }\dint\limits_{0}^{1}\left[
(\lambda +\mu )t-\lambda \right] f^{\prime }(tb+(1-t)a)dt
\end{equation*}
\end{lemma}

\begin{proof}
integration by parts we have%
\begin{eqnarray*}
I &=&\dint\limits_{0}^{1}\left[ (\lambda +\mu )t-\lambda \right] f^{\prime
}(tb+(1-t)a)dt \\
&=&\left[ (\lambda +\mu )t-\lambda \right] \frac{f(tb+(1-t)a)}{b-a}\mid
_{0}^{1}-\frac{\lambda +\mu }{b-a}\dint\limits_{0}^{1}f(tb+(1-t)a)dt \\
&=&\frac{\lambda f(a)+\mu f(b)}{b-a}-\frac{\lambda +\mu }{b-a}%
\dint\limits_{0}^{1}f(tb+(1-t)a)dt
\end{eqnarray*}

Setting $x=tb+(1-t)a$, and $dx=\left( b-a\right) dt$ gives%
\begin{equation*}
I=\frac{\lambda f(a)+\mu f(b)}{b-a}-\frac{\lambda +\mu }{\left( b-a\right)
^{2}}\dint\limits_{a}^{b}f(x)dx.
\end{equation*}%
Therefore,%
\begin{equation*}
\left( \frac{b-a}{\lambda +\mu }\right) I=\frac{\lambda f(a)+\mu f(b)}{%
\lambda +\mu }-\frac{1}{b-a}\dint\limits_{a}^{b}f(x)dx
\end{equation*}

which completes the proof.
\end{proof}

\bigskip

The next theorem gives a new refinement of the upper Hermite-Hadamard
inequality for $(\alpha ,m)$-convex functions.

\begin{theorem}
\label{1.1}Let $f:I\subset \left[ 0,\infty \right) \mathbb{\rightarrow R}$
be a differentiable mapping on $I^{\circ }$ such that $f^{\prime }\in
L[a,b], $ where $a,b\in I^{\circ }$ with $a<b.$ If $\left\vert f^{\prime
}\right\vert ^{q}$ is $(\alpha ,m)$-convex on $[a,b]$, for some fixed $%
\left( \alpha ,m\right) \in \left( 0,1\right] ^{2},~\lambda ,\mu \in \left[
0,\infty \right) $ with $\lambda +\mu >0,$ and $q\geq 1$, then the following
inequality holds:%
\begin{eqnarray}
&&\left\vert \frac{\lambda f(a)+\mu f(b)}{\lambda +\mu }-\frac{1}{b-a}%
\dint\limits_{a}^{b}f(x)dx\right\vert \leq \frac{b-a}{\lambda +\mu }\left( 
\frac{\lambda ^{2}+\mu ^{2}}{2\left( \lambda +\mu \right) }\right) ^{\frac{%
q-1}{q}}  \label{2-2} \\
&&\times \min \left\{ \left( \gamma _{1}\left\vert f^{\prime }(b)\right\vert
^{q}+m\gamma _{2}\left\vert f^{\prime }(\frac{a}{m})\right\vert ^{q}\right)
^{\frac{1}{q}},~\left( \gamma _{3}\left\vert f^{\prime }(a)\right\vert
^{q}+m\gamma _{4}\left\vert f^{\prime }(\frac{b}{m})\right\vert ^{q}\right)
^{\frac{1}{q}}\right\}  \notag
\end{eqnarray}%
where%
\begin{equation*}
\gamma _{1}=\frac{1}{\left( \alpha +1\right) \left( \alpha +2\right) }\left[ 
\frac{2\lambda ^{\alpha +2}}{\left( \lambda +\mu \right) ^{\alpha +1}}%
+\left( \alpha +1\right) \mu -\lambda \right] ,~\gamma _{2}=\frac{\lambda
^{2}+\mu ^{2}}{2\left( \lambda +\mu \right) }-\gamma _{1},
\end{equation*}%
and%
\begin{equation*}
\gamma _{3}=\frac{1}{\left( \alpha +1\right) \left( \alpha +2\right) }\left[ 
\frac{2\mu ^{\alpha +2}}{\left( \lambda +\mu \right) ^{\alpha +1}}+\left(
\alpha +1\right) \lambda -\mu \right] \text{, }\gamma _{4}=\frac{\lambda
^{2}+\mu ^{2}}{2\left( \lambda +\mu \right) }-\gamma _{3}.
\end{equation*}

\begin{proof}
Suppose that $q=1$. From Lemma \ref{2.1} and using the $(\alpha ,m)$%
-convexity $\left\vert f^{\prime }\right\vert $, we have%
\begin{eqnarray*}
&&\left\vert \frac{\lambda f(a)+\mu f(b)}{\lambda +\mu }-\frac{1}{b-a}%
\dint\limits_{a}^{b}f(x)dx\right\vert \\
&\leq &\frac{b-a}{\lambda +\mu }\dint\limits_{0}^{1}\left\vert (\lambda +\mu
)t-\lambda \right\vert \left\vert f^{\prime }(tb+(1-t)a)\right\vert dt \\
&\leq &\frac{b-a}{\lambda +\mu }\dint\limits_{0}^{1}\left\vert (\lambda +\mu
)t-\lambda \right\vert \left[ t^{\alpha }\left\vert f^{\prime
}(b)\right\vert +m(1-t^{\alpha })\right] \left\vert f^{\prime }(\frac{a}{m}%
)\right\vert )dt \\
&=&\frac{b-a}{\lambda +\mu }\dint\limits_{0}^{1}\left\vert (\lambda +\mu
)t-\lambda \right\vert t^{\alpha }\left\vert f^{\prime }(b)\right\vert
+m(1-t^{\alpha })\left\vert (\lambda +\mu )t-\lambda \right\vert \left\vert
f^{\prime }(\frac{a}{m})\right\vert dt
\end{eqnarray*}

We have 
\begin{eqnarray*}
\dint\limits_{0}^{1}\left\vert (\lambda +\mu )t-\lambda \right\vert
t^{\alpha }dt &=&\dint\limits_{0}^{\frac{\lambda }{\lambda +\mu }}\left[
\lambda -(\lambda +\mu )t\right] t^{\alpha }dt+\dint\limits_{\frac{\lambda }{%
\lambda +\mu }}^{1}\left[ (\lambda +\mu )t-\lambda \right] t^{\alpha }dt \\
&=&\frac{1}{\left( \alpha +1\right) \left( \alpha +2\right) }\left[ \frac{%
2\lambda ^{\alpha +2}}{\left( \lambda +\mu \right) ^{\alpha +1}}+\left(
\alpha +1\right) \mu -\lambda \right] =\gamma _{1}
\end{eqnarray*}

and 
\begin{equation*}
\dint\limits_{0}^{1}\left\vert (\lambda +\mu )t-\lambda \right\vert \left(
1-t^{\alpha }\right) dt=\frac{\lambda ^{2}+\mu ^{2}}{2\left( \lambda +\mu
\right) }-\gamma _{1}=\gamma _{2},
\end{equation*}

hence%
\begin{equation*}
\left\vert \frac{\lambda f(a)+\mu f(b)}{\lambda +\mu }-\frac{1}{b-a}%
\dint\limits_{a}^{b}f(x)dx\right\vert \leq \frac{b-a}{\lambda +\mu }\left(
\gamma _{1}\left\vert f^{\prime }(b)\right\vert +m\gamma _{2}\left\vert
f^{\prime }(\frac{a}{m})\right\vert \right) .
\end{equation*}

Since%
\begin{equation*}
\dint\limits_{0}^{1}\left\vert (\lambda +\mu )t-\lambda \right\vert
\left\vert f^{\prime }(tb+(1-t)a)\right\vert
dt=\dint\limits_{0}^{1}\left\vert (\lambda +\mu )t-\mu \right\vert
\left\vert f^{\prime }(ta+(1-t)b)\right\vert dt
\end{equation*}

Analogously we obtain%
\begin{equation*}
\left\vert \frac{\lambda f(a)+\mu f(b)}{\lambda +\mu }-\frac{1}{b-a}%
\dint\limits_{a}^{b}f(x)dx\right\vert \leq \frac{b-a}{\lambda +\mu }\left(
\gamma _{3}\left\vert f^{\prime }(a)\right\vert +m\gamma _{4}\left\vert
f^{\prime }(\frac{b}{m})\right\vert \right) ,
\end{equation*}

where%
\begin{equation*}
\gamma _{3}=\frac{1}{\left( \alpha +1\right) \left( \alpha +2\right) }\left[ 
\frac{2\lambda ^{\alpha +2}}{\left( \lambda +\mu \right) ^{\alpha +1}}%
+\left( \alpha +1\right) \lambda -\mu \right] \text{ and }\gamma _{4}=\frac{%
\lambda ^{2}+\mu ^{2}}{2\left( \lambda +\mu \right) }-\gamma _{3}
\end{equation*}

which copmletes the proof for this case.

Suppose now that $q\in \left( 1,\infty \right) $. From Lemma \ref{2.1} and
using the H\"{o}lder's integral inequality, we have%
\begin{eqnarray}
&&\dint\limits_{0}^{1}\left\vert (\lambda +\mu )t-\lambda \right\vert
\left\vert f^{\prime }(tb+(1-t)a)\right\vert dt  \label{2} \\
&\leq &\frac{b-a}{\lambda +\mu }\left( \dint\limits_{0}^{1}\left\vert
(\lambda +\mu )t-\lambda \right\vert dt\right) ^{\frac{q-1}{q}}\left(
\dint\limits_{0}^{1}\left\vert (\lambda +\mu )t-\lambda \right\vert
\left\vert f^{\prime }(tb+(1-t)a)\right\vert ^{q}dt\right) ^{\frac{1}{q}} 
\notag
\end{eqnarray}

Since $\left\vert f^{\prime }\right\vert ^{q}$ is $(\alpha ,m)$-convex on $%
[a,b]$, we know that for every $t\in \left[ 0,1\right] $%
\begin{equation}
\left\vert f^{\prime }(tb+m(1-t)\frac{a}{m})\right\vert ^{q}\leq t^{\alpha
}\left\vert f^{\prime }(b)\right\vert ^{q}+m(1-t^{\alpha })\left\vert
f^{\prime }(\frac{a}{m})\right\vert ^{q}.  \label{3}
\end{equation}%
From the inequalities (\ref{1}), (\ref{2}) and (\ref{3}), we have%
\begin{equation*}
\left\vert \frac{\lambda f(a)+\mu f(b)}{\lambda +\mu }-\frac{1}{b-a}%
\dint\limits_{a}^{b}f(x)dx\right\vert \leq \frac{b-a}{\lambda +\mu }\left( 
\frac{\lambda ^{2}+\mu ^{2}}{2\left( \lambda +\mu \right) }\right) ^{\frac{%
q-1}{q}}\left( \gamma _{1}\left\vert f^{\prime }(b)\right\vert ^{q}+m\gamma
_{2}\left\vert f^{\prime }(\frac{a}{m})\right\vert ^{q}\right) ^{\frac{1}{q}}
\end{equation*}

and analogously%
\begin{equation*}
\left\vert \frac{\lambda f(a)+\mu f(b)}{\lambda +\mu }-\frac{1}{b-a}%
\dint\limits_{a}^{b}f(x)dx\right\vert \leq \frac{b-a}{\lambda +\mu }\left( 
\frac{\lambda ^{2}+\mu ^{2}}{2\left( \lambda +\mu \right) }\right) ^{\frac{%
q-1}{q}}\left( \gamma _{3}\left\vert f^{\prime }(a)\right\vert ^{q}+m\gamma
_{4}\left\vert f^{\prime }(\frac{b}{m})\right\vert ^{q}\right) ^{\frac{1}{q}}
\end{equation*}

which completes the proof.
\end{proof}

\begin{corollary}
\label{2.a}Suppose that all the assumptions of Theorem\ref{1.1} are
satisfied,

\begin{enumerate}
\item In the inequality (\ref{2-2}) If we choose $\lambda =\mu $ , we obtain
the inequality in (\ref{1-4}).

\item In the inequality (\ref{2-2}) If we choose $\lambda =\mu ,$ $m=1,~q=1~$%
and $\alpha =1$ we obtain the inequality in (\ref{1-2}).

\item In the inequality (\ref{2-2}) If we choose $m=\alpha =1$ we have%
\begin{eqnarray}
&&\left\vert \frac{\lambda f(a)+\mu f(b)}{\lambda +\mu }-\frac{1}{b-a}%
\dint\limits_{a}^{b}f(x)dx\right\vert  \label{2-a} \\
&\leq &\frac{b-a}{\lambda +\mu }\left( \frac{\lambda ^{2}+\mu ^{2}}{2\left(
\lambda +\mu \right) }\right) ^{\frac{q-1}{q}}  \notag \\
&&\times \min \left\{ \left( \gamma _{1}\left\vert f^{\prime }(b)\right\vert
^{q}+\gamma _{2}\left\vert f^{\prime }(a)\right\vert ^{q}\right) ^{\frac{1}{q%
}},~\left( \gamma _{3}\left\vert f^{\prime }(a)\right\vert ^{q}+\gamma
_{4}\left\vert f^{\prime }(b)\right\vert ^{q}\right) ^{\frac{1}{q}}\right\} 
\notag
\end{eqnarray}
\end{enumerate}
\end{corollary}
\end{theorem}

where%
\begin{equation*}
\gamma _{1}=\frac{1}{6}\left[ \frac{2\lambda ^{3}}{\left( \lambda +\mu
\right) ^{2}}+2\mu -\lambda \right] ,~\gamma _{2}=\frac{\lambda ^{2}+\mu ^{2}%
}{2\left( \lambda +\mu \right) }-\gamma _{1},
\end{equation*}%
and%
\begin{equation*}
\gamma _{3}=\frac{1}{6}\left[ \frac{2\mu ^{3}}{\left( \lambda +\mu \right)
^{2}}+2\lambda -\mu \right] ,\gamma _{4}=\frac{\lambda ^{2}+\mu ^{2}}{%
2\left( \lambda +\mu \right) }-\gamma _{3}.
\end{equation*}

\begin{theorem}
\label{2.1.1}Let $f:I\subset \left[ 0,\infty \right) \mathbb{\rightarrow R}$
be a differentiable mapping on $I^{\circ }$ such that $f^{\prime }\in
L[a,b], $ where $a,b\in I^{\circ }$ with $a<b.$ If $\left\vert f^{\prime
}\right\vert ^{q}$ is $(\alpha ,m)$-convex on $[a,b]$, for some fixed $%
\left( \alpha ,m\right) \in \left( 0,1\right] ^{2},~\lambda ,\mu \in \left[
0,\infty \right) $ with $\lambda +\mu >0,$ and $q>1$, then the following
inequality holds:%
\begin{eqnarray}
&&\left\vert \frac{\lambda f(a)+\mu f(b)}{\lambda +\mu }-\frac{1}{b-a}%
\dint\limits_{a}^{b}f(x)dx\right\vert \leq ~\frac{b-a}{\left( \lambda +\mu
\right) ^{2}}  \label{2-4} \\
&&~\times \left( \frac{1}{p+1}\right) ^{\frac{1}{p}}\left[ \lambda
^{2}M_{1}^{\frac{1}{q}}+\mu ^{2}M_{2}^{\frac{1}{q}}\right]  \notag
\end{eqnarray}%
where%
\begin{eqnarray*}
M_{1} &=&\min \left\{ \frac{\left\vert f^{\prime }(a)\right\vert ^{q}+\alpha
m\left\vert f^{\prime }\left( \frac{\lambda b+\mu a}{m(\lambda +\mu )}%
\right) \right\vert ^{q}}{\alpha +1},~\frac{\left\vert f^{\prime }(\frac{%
\lambda b+\mu a}{\lambda +\mu })\right\vert ^{q}+\alpha m\left\vert
f^{\prime }\left( \frac{a}{m}\right) \right\vert ^{q}}{\alpha +1}\right\} \\
M_{2} &=&\min \left\{ \frac{\left\vert f^{\prime }(b)\right\vert ^{q}+\alpha
m\left\vert f^{\prime }\left( \frac{\lambda b+\mu a}{m(\lambda +\mu )}%
\right) \right\vert ^{q}}{\alpha +1},~\frac{\left\vert f^{\prime }(\frac{%
\lambda b+\mu a}{\lambda +\mu })\right\vert ^{q}+\alpha m\left\vert
f^{\prime }\left( \frac{b}{m}\right) \right\vert ^{q}}{\alpha +1}\right\}
\end{eqnarray*}%
and \ $\frac{1}{p}+\frac{1}{q}=1.$
\end{theorem}

\begin{proof}
From Lemma \ref{2.1} and using the H\"{o}lder inequality, we have%
\begin{eqnarray*}
&&\left\vert \frac{\lambda f(a)+\mu f(b)}{\lambda +\mu }-\frac{1}{b-a}%
\dint\limits_{a}^{b}f(x)dx\right\vert \\
&\leq &\frac{b-a}{\lambda +\mu }\left( \dint\limits_{0}^{\frac{\lambda }{%
\lambda +\mu }}\left[ \lambda -(\lambda +\mu )t\right] ^{p}dt\right) ^{\frac{%
1}{p}}\left( \dint\limits_{0}^{\frac{\lambda }{\lambda +\mu }}\left\vert
f^{\prime }(tb+(1-t)a)\right\vert ^{q}~dt\right) ^{\frac{1}{q}} \\
&&+\frac{b-a}{\lambda +\mu }\left( \dint\limits_{\frac{\lambda }{\lambda
+\mu }}^{1}\left[ (\lambda +\mu )t-\lambda \right] ^{p}dt\right) ^{\frac{1}{p%
}}\left( \dint\limits_{\frac{\lambda }{\lambda +\mu }}^{1}\left\vert
f^{\prime }(tb+(1-t)a)\right\vert ^{q}~dt\right) ^{\frac{1}{q}} \\
&\leq &\frac{b-a}{\lambda +\mu }\left[ \left( \frac{\lambda ^{p+1}}{%
(p+1)(\lambda +\mu )}\right) ^{\frac{1}{p}}\left( \frac{\lambda }{\lambda
+\mu }M_{1}\right) ^{\frac{1}{q}}+\left( \frac{\mu ^{p+1}}{(p+1)(\lambda
+\mu )}\right) ^{\frac{1}{p}}\left( \frac{\mu }{\lambda +\mu }M_{2}\right) ^{%
\frac{1}{q}}\right] \\
&=&\frac{b-a}{\lambda +\mu }\times \left( \frac{1}{p+1}\right) ^{\frac{1}{p}}%
\left[ \frac{\lambda ^{2}}{\lambda +\mu }M_{1}^{\frac{1}{q}}+\frac{\mu ^{2}}{%
\lambda +\mu }M_{2}^{\frac{1}{q}}\right]
\end{eqnarray*}%
where we use the fact that%
\begin{eqnarray*}
\dint\limits_{0}^{\frac{\lambda }{\lambda +\mu }}\left[ \lambda -(\lambda
+\mu )t\right] ^{p}dt &=&\frac{\lambda ^{p+1}}{(p+1)(\lambda +\mu )}, \\
\dint\limits_{\frac{\lambda }{\lambda +\mu }}^{1}\left[ (\lambda +\mu
)t-\lambda \right] ^{p} &=&\frac{\mu ^{p+1}}{(p+1)(\lambda +\mu )}
\end{eqnarray*}%
and by Theorem\ref{A.1} we get%
\begin{eqnarray*}
&&\frac{\lambda +\mu }{\lambda }\dint\limits_{0}^{\frac{\lambda }{\lambda
+\mu }}\left\vert f^{\prime }(tb+(1-t)a)\right\vert ^{q}~dt \\
&=&\frac{1}{\frac{\lambda }{\lambda +\mu }(b-a)}\dint\limits_{a}^{\frac{%
\lambda b+\mu a}{\lambda +\mu }}\left\vert f^{\prime }(x)\right\vert ^{q}~dx
\\
&\leq &\min \left\{ \frac{\left\vert f^{\prime }(a)\right\vert ^{q}+\alpha
m\left\vert f^{\prime }\left( \frac{\lambda b+\mu a}{m(\lambda +\mu )}%
\right) \right\vert ^{q}}{\alpha +1},~\frac{\left\vert f^{\prime }(\frac{%
\lambda b+\mu a}{\lambda +\mu })\right\vert ^{q}+\alpha m\left\vert
f^{\prime }\left( \frac{a}{m}\right) \right\vert ^{q}}{\alpha +1}\right\} ,
\\
&&\frac{\lambda +\mu }{\mu }\dint\limits_{\frac{\lambda }{\lambda +\mu }%
}^{1}\left\vert f^{\prime }(tb+(1-t)a)\right\vert ^{q}~dt \\
&=&\frac{1}{\frac{\mu }{\lambda +\mu }(b-a)}\dint\limits_{\frac{\lambda
b+\mu a}{\lambda +\mu }}^{b}\left\vert f^{\prime }(x)\right\vert ^{q}~dx \\
&\leq &\min \left\{ \frac{\left\vert f^{\prime }(b)\right\vert ^{q}+\alpha
m\left\vert f^{\prime }\left( \frac{\lambda b+\mu a}{m(\lambda +\mu )}%
\right) \right\vert ^{q}}{\alpha +1},~\frac{\left\vert f^{\prime }(\frac{%
\lambda b+\mu a}{\lambda +\mu })\right\vert ^{q}+\alpha m\left\vert
f^{\prime }\left( \frac{b}{m}\right) \right\vert ^{q}}{\alpha +1}\right\} .
\end{eqnarray*}%
which completes the proof.
\end{proof}

\begin{corollary}
\label{2.b}Suppose that all the assumptions of Theorem\ref{2.1.1} are
satisfied, in this case:
\end{corollary}

\begin{enumerate}
\item In the inequality (\ref{2-4}) if we choose $\lambda =\mu $ and $\alpha
=1$ we obtain the inequality in (\ref{B1}).

\item In the inequality (\ref{2-4}) if we choose $m=\alpha =1$ we have\qquad 
\begin{equation}
\left\vert \frac{\lambda f(a)+\mu f(b)}{\lambda +\mu }-\frac{1}{b-a}%
\dint\limits_{a}^{b}f(x)dx\right\vert \leq ~\frac{b-a}{\left( \lambda +\mu
\right) ^{2}}\left( \frac{1}{p+1}\right) ^{\frac{1}{p}}\left[ \lambda
^{2}M_{1}^{\frac{1}{q}}+\mu ^{2}M_{2}^{\frac{1}{q}}\right]  \label{2-b}
\end{equation}%
where%
\begin{equation*}
M_{1}=\frac{\left\vert f^{\prime }(a)\right\vert ^{q}+\left\vert f^{\prime
}\left( \frac{\lambda b+\mu a}{\lambda +\mu }\right) \right\vert ^{q}}{2}%
,~M_{2}=\frac{\left\vert f^{\prime }(b)\right\vert ^{q}+\left\vert f^{\prime
}\left( \frac{\lambda b+\mu a}{\lambda +\mu }\right) \right\vert ^{q}}{2}
\end{equation*}
\end{enumerate}

\begin{theorem}
\label{2.2}Let $f:I\subset \left[ 0,\infty \right) \mathbb{\rightarrow R}$
be a differentiable mapping on $I^{\circ }$ such that $f^{\prime }\in
L[a,b], $ where $a,b\in I^{\circ }$ with $a<b.$ If $\left\vert f^{\prime
}\right\vert ^{q}$ is $(\alpha ,m)$-convex on $[a,b]$, for some fixed $%
\left( \alpha ,m\right) \in \left( 0,1\right] ^{2},~\lambda ,\mu \in \left[
0,\infty \right) $ with $\lambda +\mu >0,$ and $q>1$, then the following
inequality holds:%
\begin{eqnarray}
&&\left\vert \frac{\lambda f(a)+\mu f(b)}{\lambda +\mu }-\frac{1}{b-a}%
\dint\limits_{a}^{b}f(x)dx\right\vert  \label{2-5} \\
&\leq &\frac{b-a}{\lambda +\mu }\left( \frac{\lambda ^{p+1}+\mu ^{p+1}}{%
(p+1)(\lambda +\mu )}\right) ^{\frac{1}{p}}\left( \frac{1}{\alpha +1}\right)
^{\frac{1}{q}}\min \left\{ K_{1}^{\frac{1}{q}},K_{2}^{\frac{1}{q}}\right\} ,
\notag
\end{eqnarray}%
where%
\begin{eqnarray*}
K_{1} &=&\left\vert f^{\prime }(b)\right\vert ^{q}+m\alpha \left\vert
f^{\prime }(\frac{a}{m})\right\vert ^{q} \\
K_{2} &=&\left\vert f^{\prime }(a)\right\vert ^{q}+m\alpha \left\vert
f^{\prime }(\frac{b}{m})\right\vert ^{q}
\end{eqnarray*}%
and \ $\frac{1}{p}+\frac{1}{q}=1.$
\end{theorem}

\begin{proof}
From Lemma \ref{2.1} and using the H\"{o}lder's integral inequality, we have%
\begin{eqnarray*}
&&\left\vert \frac{\lambda f(a)+\mu f(b)}{\lambda +\mu }-\frac{1}{b-a}%
\dint\limits_{a}^{b}f(x)dx\right\vert \\
&\leq &\dint\limits_{0}^{1}\left\vert (\lambda +\mu )t-\lambda \right\vert
\left\vert f^{\prime }(tb+(1-t)a)\right\vert dt \\
&\leq &\frac{b-a}{\lambda +\mu }\left( \dint\limits_{0}^{1}\left\vert
(\lambda +\mu )t-\lambda \right\vert ^{p}dt\right) ^{\frac{1}{p}}\left(
\dint\limits_{0}^{1}\left\vert f^{\prime }(tb+(1-t)a)\right\vert
^{q}dt\right) ^{\frac{1}{q}} \\
&\leq &\frac{b-a}{\lambda +\mu }\left( \frac{\lambda ^{p+1}+\mu ^{p+1}}{%
(p+1)(\lambda +\mu )}\right) ^{\frac{1}{p}}\left(
\dint\limits_{0}^{1}t^{\alpha }\left\vert f^{\prime }(b)\right\vert
^{q}+m(1-t^{\alpha })\left\vert f^{\prime }(\frac{a}{m})\right\vert
^{q}dt\right) ^{\frac{1}{q}} \\
&=&\frac{b-a}{\lambda +\mu }\left( \frac{\lambda ^{p+1}+\mu ^{p+1}}{%
(p+1)(\lambda +\mu )}\right) ^{\frac{1}{p}}\left( \frac{1}{\alpha +1}\right)
^{\frac{1}{q}}\left( \left\vert f^{\prime }(b)\right\vert ^{q}+m\alpha
\left\vert f^{\prime }(\frac{a}{m})\right\vert ^{q}\right) ^{\frac{1}{q}}.
\end{eqnarray*}%
Analogously we obtain%
\begin{eqnarray*}
&&\left\vert \frac{\lambda f(a)+\mu f(b)}{\lambda +\mu }-\frac{1}{b-a}%
\dint\limits_{a}^{b}f(x)dx\right\vert \\
&\leq &\frac{b-a}{\lambda +\mu }\left( \frac{\lambda ^{p+1}+\mu ^{p+1}}{%
(p+1)(\lambda +\mu )}\right) ^{\frac{1}{p}}\left( \frac{1}{\alpha +1}\right)
^{\frac{1}{q}}\left( \left\vert f^{\prime }(a)\right\vert ^{q}+m\alpha
\left\vert f^{\prime }(\frac{b}{m})\right\vert ^{q}\right) ^{\frac{1}{q}}.
\end{eqnarray*}
\end{proof}

\begin{corollary}
\label{2.c}Suppose that all the assumptions of Theorem\ref{2.2} are
satisfied, in this case:
\end{corollary}

\begin{enumerate}
\item In the inequality (\ref{2-5}) if we choose $\lambda =\mu $, then the
following inequality holds:%
\begin{eqnarray*}
&&\left\vert \frac{f(a)+f(b)}{2}-\frac{1}{b-a}\dint\limits_{a}^{b}f(x)dx%
\right\vert \\
&\leq &\frac{b-a}{2}\left( \frac{1}{p+1}\right) ^{\frac{1}{p}}\left( \frac{1%
}{\alpha +1}\right) ^{\frac{1}{q}}\min \left\{ K_{1}^{\frac{1}{q}},K_{2}^{%
\frac{1}{q}}\right\} ,
\end{eqnarray*}

\item In the inequality (\ref{2-5}) if we choose $m=\alpha =1$, we have%
\begin{eqnarray}
&&\left\vert \frac{\lambda f(a)+\mu f(b)}{\lambda +\mu }-\frac{1}{b-a}%
\dint\limits_{a}^{b}f(x)dx\right\vert  \label{2-c} \\
&\leq &\frac{b-a}{\lambda +\mu }\left( \frac{\lambda ^{p+1}+\mu ^{p+1}}{%
\lambda +\mu }\right) ^{\frac{1}{p}}\left( \frac{1}{p+1}\right) ^{\frac{1}{p}%
}\left( \frac{1}{2}\right) ^{\frac{1}{q}}\left( \left\vert f^{\prime
}(a)\right\vert ^{q}+\left\vert f^{\prime }(b)\right\vert ^{q}\right) ^{%
\frac{1}{q}}  \notag
\end{eqnarray}
\end{enumerate}

\section{Some applications for special means}

Let us recall the following special means of two nonnegative number $a,b$
with $b\geq a$ and $\alpha \in \left[ 0,1\right] :$

\begin{enumerate}
\item The weighted arithmetic mean%
\begin{equation*}
A_{\alpha }\left( a,b\right) :=\alpha a+(1-\alpha )b,~a,b\geq 0.
\end{equation*}

\item The unweighted arithmetic mean%
\begin{equation*}
A\left( a,b\right) :=\frac{a+b}{2},~a,b\geq 0.
\end{equation*}

\item The weighted harmonic mean%
\begin{equation*}
H_{\alpha }\left( a,b\right) :=\left( \frac{\alpha }{a}+\frac{1-\alpha }{b}%
\right) ^{-1},\ \ a,b>0.
\end{equation*}

\item The unweighted harmonic mean%
\begin{equation*}
H\left( a,b\right) :=\frac{2ab}{a+b},\ \ a,b>0.
\end{equation*}

\item The Logarithmic mean%
\begin{equation*}
L\left( a,b\right) :=\left\{ 
\begin{array}{cc}
\frac{b-a}{\ln b-\ln a} & if\ a\neq b \\ 
b & if\ a=b%
\end{array}%
\right. ,\ \ a,b>0.
\end{equation*}

\item The p-Logarithmic mean%
\begin{equation*}
L_{p}\left( a,b\right) :=\left\{ 
\begin{array}{cc}
\left( \frac{b^{p+1}-a^{p+1}}{(p+1)(b-a)}\right) ^{\frac{1}{p}} & if\ a\neq b
\\ 
b & if\ a=b%
\end{array}%
\right. ,\ \ a,b>0,\ p\in 
%TCIMACRO{\U{2124} }%
%BeginExpansion
\mathbb{Z}
%EndExpansion
\backslash \left\{ -1,0\right\} .
\end{equation*}
\end{enumerate}

\begin{proposition}
Let $a,b\in 
%TCIMACRO{\U{211d} }%
%BeginExpansion
\mathbb{R}
%EndExpansion
$ with $0<a<b$ and $n\in 
%TCIMACRO{\U{2124} }%
%BeginExpansion
\mathbb{Z}
%EndExpansion
,\ \left\vert n\right\vert \geq 2.$ Then, we have the following inequality:%
\begin{eqnarray*}
\left\vert A_{\frac{\lambda }{\lambda +\mu }}\left( a^{n},b^{n}\right)
-L_{n}^{n}\left( a,b\right) \right\vert  &\leq &\frac{b-a}{\lambda +\mu }%
\left( \frac{\lambda ^{2}+\mu ^{2}}{2\left( \lambda +\mu \right) }\right) ^{%
\frac{q-1}{q}} \\
&&\times \left\vert n\right\vert \min \left\{ \left( \gamma
_{1}b^{q(n-1)}+\gamma _{2}a^{^{q(n-1)}}\right) ^{\frac{1}{q}},~\left( \gamma
_{3}a^{^{q(n-1)}}+\gamma _{4}b^{^{q(n-1)}}\right) ^{\frac{1}{q}}\right\} 
\end{eqnarray*}%
where%
\begin{equation*}
\gamma _{1}=\frac{1}{6}\left[ \frac{2\lambda ^{3}}{\left( \lambda +\mu
\right) ^{2}}+2\mu -\lambda \right] ,~\gamma _{2}=\frac{\lambda ^{2}+\mu ^{2}%
}{2\left( \lambda +\mu \right) }-\gamma _{1},
\end{equation*}%
\begin{equation*}
\gamma _{3}=\frac{1}{6}\left[ \frac{2\mu ^{3}}{\left( \lambda +\mu \right)
^{2}}+2\lambda -\mu \right] ,\gamma _{4}=\frac{\lambda ^{2}+\mu ^{2}}{%
2\left( \lambda +\mu \right) }-\gamma _{3},
\end{equation*}%
$\lambda ,\mu \in \left[ 0,\infty \right) $ with $\lambda +\mu >0$ and $%
q\geq 1.$
\end{proposition}

\begin{proof}
The assertion follows from inequality (\ref{2-a}) in Corollary \ref{2.a},
for $f:\left( 0,\infty \right) \mathbb{\rightarrow R},\ f(x)=x^{n}$
\end{proof}

\begin{proposition}
Let $a,b\in 
%TCIMACRO{\U{211d} }%
%BeginExpansion
\mathbb{R}
%EndExpansion
$ with $0<a<b$ and $n\in 
%TCIMACRO{\U{2124} }%
%BeginExpansion
\mathbb{Z}
%EndExpansion
,\ \left\vert n\right\vert \geq 2.$ Then, we have the following inequality:%
\begin{equation*}
\left\vert A_{\frac{\lambda }{\lambda +\mu }}\left( a^{n},b^{n}\right)
-L_{n}^{n}\left( a,b\right) \right\vert \leq \frac{b-a}{\left( \lambda +\mu
\right) ^{2}}\left( \frac{1}{p+1}\right) ^{\frac{1}{p}}\left\vert
n\right\vert \left[ \lambda ^{2}M_{1}^{\frac{1}{q}}+\mu ^{2}M_{2}^{\frac{1}{q%
}}\right] 
\end{equation*}%
where%
\begin{equation*}
M_{1}=A\left( a^{(n-1)q},A_{\frac{\lambda }{\lambda +\mu }}^{(n-1)q}\left(
b,a\right) \right) 
\end{equation*}%
\begin{equation*}
M_{2}=A\left( b^{(n-1)q},A_{\frac{\lambda }{\lambda +\mu }}^{(n-1)q}\left(
b,a\right) \right) ,
\end{equation*}%
$\ \lambda ,\mu \in \left[ 0,\infty \right) $ with $\lambda +\mu >0,$ and $%
q>1$ with $\frac{1}{p}+\frac{1}{q}=1.$
\end{proposition}

\begin{proof}
The assertion follows from inequality (\ref{2-b}) in Corollary \ref{2.b},
for $f:\left( 0,\infty \right) \mathbb{\rightarrow R},\ f(x)=x^{n}$
\end{proof}

\begin{proposition}
Let $a,b\in 
%TCIMACRO{\U{211d} }%
%BeginExpansion
\mathbb{R}
%EndExpansion
$ with $0<a<b$ and $n\in 
%TCIMACRO{\U{2124} }%
%BeginExpansion
\mathbb{Z}
%EndExpansion
,\ \left\vert n\right\vert \geq 2.$ Then, we have the following inequality:%
\begin{eqnarray*}
&&\left\vert A_{\frac{\lambda }{\lambda +\mu }}\left( a^{n},b^{n}\right)
-L_{n}^{n}\left( a,b\right) \right\vert  \\
&\leq &\frac{b-a}{\lambda +\mu }\left( \frac{\lambda ^{p+1}+\mu ^{p+1}}{%
\lambda +\mu }\right) ^{\frac{1}{p}}\left( \frac{1}{p+1}\right) ^{\frac{1}{p}%
}\left\vert n\right\vert A^{\frac{1}{q}}\left( a^{(n-1)q},b^{(n-1)q}\right) 
\end{eqnarray*}%
$\lambda ,\mu \in \left[ 0,\infty \right) $ with $\lambda +\mu >0,$ and $q>1$
with $\frac{1}{p}+\frac{1}{q}=1.$
\end{proposition}

\begin{proof}
The assertion follows from inequality (\ref{2-c}) in Corollary \ref{2.c},
for $f:\left( 0,\infty \right) \mathbb{\rightarrow R},\ f(x)=\frac{1}{x}$
\end{proof}

\begin{proposition}
Let $a,b\in 
%TCIMACRO{\U{211d} }%
%BeginExpansion
\mathbb{R}
%EndExpansion
$ with $0<a<b$ $.$Then we have the following inequality:%
\begin{eqnarray*}
&&\left\vert H_{\frac{\lambda }{\lambda +\mu }}^{-1}\left( a,b\right)
-L^{-1}\left( a,b\right) \right\vert  \\
&\leq &\frac{b-a}{\lambda +\mu }\left( \frac{\lambda ^{2}+\mu ^{2}}{2\left(
\lambda +\mu \right) }\right) ^{\frac{q-1}{q}} \\
&&\times \min \left\{ \left( \gamma _{1}\frac{1}{b^{2q}}+\gamma _{2}\frac{1}{%
a^{2q}}\right) ^{\frac{1}{q}},~\left( \gamma _{3}\frac{1}{a^{2q}}+\gamma _{4}%
\frac{1}{b^{2q}}\right) ^{\frac{1}{q}}\right\} 
\end{eqnarray*}%
where%
\begin{equation*}
\gamma _{1}=\frac{1}{6}\left[ \frac{2\lambda ^{3}}{\left( \lambda +\mu
\right) ^{2}}+2\mu -\lambda \right] ,~\gamma _{2}=\frac{\lambda ^{2}+\mu ^{2}%
}{2\left( \lambda +\mu \right) }-\gamma _{1},
\end{equation*}%
\begin{equation*}
\gamma _{3}=\frac{1}{6}\left[ \frac{2\mu ^{3}}{\left( \lambda +\mu \right)
^{2}}+2\lambda -\mu \right] ,\gamma _{4}=\frac{\lambda ^{2}+\mu ^{2}}{%
2\left( \lambda +\mu \right) }-\gamma _{3},
\end{equation*}%
$\lambda ,\mu \in \left[ 0,\infty \right) $ with $\lambda +\mu >0$ and $%
q\geq 1.$
\end{proposition}

\begin{proof}
The assertion follows from inequality (\ref{2-a}) in Corollary \ref{2.a},
for $f:\left( 0,\infty \right) \mathbb{\rightarrow R},\ f(x)=x^{n}$
\end{proof}

\begin{proposition}
Let $a,b\in 
%TCIMACRO{\U{211d} }%
%BeginExpansion
\mathbb{R}
%EndExpansion
$ with $0<a<b.$Then we have the following inequality:%
\begin{equation*}
\left\vert H_{\frac{\lambda }{\lambda +\mu }}^{-1}\left( a,b\right)
-L^{-1}\left( a,b\right) \right\vert \leq \frac{b-a}{\left( \lambda +\mu
\right) ^{2}}\left( \frac{1}{p+1}\right) ^{\frac{1}{p}}\left[ \lambda
^{2}M_{1}^{\frac{1}{q}}+\mu ^{2}M_{2}^{\frac{1}{q}}\right] 
\end{equation*}%
where%
\begin{equation*}
M_{1}=H^{-1}\left( a^{2q},A_{\frac{\lambda }{\lambda +\mu }}^{2q}\left(
b,a\right) \right) ,~M_{2}=H^{-1}\left( b^{2q},A_{\frac{\lambda }{\lambda
+\mu }}^{2q}\left( b,a\right) \right) 
\end{equation*}%
$\ \lambda ,\mu \in \left[ 0,\infty \right) $ with $\lambda +\mu >0,$ and $%
q>1$ with $\frac{1}{p}+\frac{1}{q}=1.$
\end{proposition}

\begin{proof}
The assertion follows from inequality (\ref{2-b}) in Corollary \ref{2.b},
for $f:\left( 0,\infty \right) \mathbb{\rightarrow R},\ f(x)=\frac{1}{x}$
\end{proof}

\begin{proposition}
Let $a,b\in 
%TCIMACRO{\U{211d} }%
%BeginExpansion
\mathbb{R}
%EndExpansion
$ with $0<a<b$ . Then we have the following inequality:%
\begin{eqnarray*}
&&\left\vert H_{\frac{\lambda }{\lambda +\mu }}^{-1}\left( a,b\right)
-L^{-1}\left( a,b\right) \right\vert  \\
&\leq &\frac{b-a}{\lambda +\mu }A_{\frac{\lambda }{\lambda +\mu }}^{\frac{1}{%
p}}\left( \lambda ^{p},\mu ^{p}\right) \left( \frac{1}{p+1}\right) ^{\frac{1%
}{p}}\left( \frac{1}{2}\right) ^{\frac{1}{q}}H_{\frac{1}{2}}^{\frac{-1}{q}%
}\left( a^{2q},b^{2q}\right) 
\end{eqnarray*}%
$\lambda ,\mu \in \left[ 0,\infty \right) $ with $\lambda +\mu >0,$ and $q>1$
with $\frac{1}{p}+\frac{1}{q}=1.$
\end{proposition}

\begin{proof}
The assertion follows from inequality (\ref{2-c}) in Corollary \ref{2.c},
for $f:\left( 0,\infty \right) \mathbb{\rightarrow R},\ f(x)=\frac{1}{x}$
\end{proof}

\end{document}